\documentclass{amsart}

\usepackage{hyperref}
\usepackage{tikz-cd}
\usepackage{amsfonts,amsmath,amssymb,mathrsfs, stmaryrd, amsthm}
\usepackage{enumitem}
\usepackage{cite}
\usepackage{amsrefs}
\usepackage{bm}

\title[Covariant representations of algebraic actions]{Covariant representations \\ of algebraic group actions \\ and applications}
\author{Yvann Gaudillot-Estrada}
\address{Université de Lorraine, CNRS, IECL, F-57000 Metz, France}
\email{yvann.gaudillot-estrada@univ-lorraine.fr}
\thanks{This research was supported by the projects OpART (ANR-23-CE40-0016) and CroCQG (ANR-25-CE40-5010) of the \emph{Agence Nationale de la Recherche}. It is based upon work from COST Action CaLISTA CA21109 supported by COST (European Cooperation in Science and Technology).}

\newtheorem{theorem}{Theorem}
\newtheorem{proposition}[theorem]{Proposition}
\newtheorem{lemma}[theorem]{Lemma}

\theoremstyle{definition}
\newtheorem{definition}[theorem]{Definition}
\newtheorem{remark}[theorem]{Remark}

\setlist[itemize]{leftmargin=0.7cm}

\newcommand{\dquo}{\backslash \!\! \backslash}

\begin{document}    

\begin{abstract}
    If $G$ is an algebraic affine group acting on an affine variety $X$, there is a natural notion of covariant representation for the pair $(G,X)$. In this paper, we classify the irreducible covariant representations for any such pair by adapting the Mackey machine to this algebraic setting. Next, we give applications for continuous representations of motion groups on Banach spaces and other related examples.
\end{abstract}

\maketitle

\section*{Introduction}

One of the fundamental applications of Mackey's theory of induction \cite{Mackey} is the classification of the unitary irreducible representations of semi-direct products of the form $K \ltimes V$, with $K$ a compact group and $V$ abelian. In modern terms, the reason why the Mackey machine applies to this example is the following: there is natural correspondence between unitary representations of $K \ltimes V$ and covariant representations \cites{Dana, Raeburn} of the C*-dynamical system $(C^*(V), K)$, where $C^*(V)$ denotes the maximal group C*-algebra of $V$. Thanks to this correspondence, the unitary dual of $K \ltimes V$ can be described from the orbit structure of the spectrum of $C^*(V)$, which identifies with the Pontryagin dual of $V$. When $K$ is a compact Lie group and $V$ a real vector space, one can show that this correspondence extends to admissible Banach representations of $K \ltimes V$, provided that we replace the C*-dynamical system $(C^*(V), K)$ by the action of an affine algebraic group (the complexification of $K$) on a polynomial algebra (the symmetric algebra of $\mathbb{C} \otimes_\mathbb{R} V$). Motivated by this example, the purpose of this paper is to study covariant representations associated to actions of affine algebraic group, first abstractly, then in relation with concrete example arising from Lie theory.

Let $G$ be an affine group defined over an algebraically closed field $k$ and let $X$ be an affine variety equipped with an algebraic action of $G$. We refer to $(G,X)$ as an \textit{affine action pair}. Let us simply call \textit{representation of $G$} any representation of $G$ which is the sum of its rational subrepresentations. Let us denote by $\tau$ the action by translations of $G$ on the algebra $k[X]$ of regular functions on $X$. We are interested in the following kind of representations.

\begin{definition} \label{module}
    A covariant representation of $(G,X)$ is a representation $\rho$ of $G$ on a $k[X]$-module $M$ such that
    $$\rho(g)f\psi = (\tau_gf)\rho(g)\psi \qquad (g \in G, f\in k[X], \psi \in M).$$
\end{definition}

In the present paper, we classify the irreducible covariant representations of affine action pairs when $k$ has characteristic zero (Section 2). As we will see, this can be reduced to the transitive case, for which we can use the main theorem of \cite{2authors}. We will give a new proof of the latter by following a more elementary approach.


As we mentioned above, one of our main motivation is to show that we can still apply imprimitivity arguments for non-unitary Banach representations of semi-direct products. As an application of our description of irreducible covariant representations, we classify in Section 3 the topologically completely irreducible representations of a large class of semi-direct product groups. In particular, we give a streamlined proof of Champetier--Delorme's classification of irreducible admissible representations of Cartan motion groups \cite{DelormeCh}. Note that this result has only been proved for motion groups associated to semisimple Lie groups with finite center. In what follows, we check that it extends to Cartan motion groups of arbitrary real reductive groups.

On the other hand, from the perspective of the Mackey analogy \cites{Afgoustidis1, MackeyAnalogy, MAhigson, MonkVoigt}, our recent work \cite{YGE} suggests that the analogues of Cartan motion groups for real semisimple quantum groups \cite{DCquantisation} should be affine action pairs of the form $(G^\theta, G /G^\theta)$, where $G$ is a simply connected complex semisimple group and $\theta$ a complexified Cartan involution associated to a real form of $G$. In Section 4, we make explicit the classification of irreducible covariant representations for these action pairs. To do so, we prove an analogue of Chevalley restriction theorem for $G^\theta$-bi-invariant regular functions on $G$.

\section{Preliminaries}

Our first goal is to summarize elements of algebraic geometry, essentially from the theory of algebraic groups, that we will use later. For the fist three subsections, we fix an algebraically closed field $k$ and consider varieties defined over $k$ which are not necessary irreducible.  If $X$ is a variety, we denote by $\mathscr{O}_X$ its structure sheaf and we write $k[X]$ for the algebra of regular functions on $X$.

\subsection{Modules over the ring of regular functions}

Let $X$ be an affine variety. We denote by $\mathrm{Sh}$ the functor associating to a $k[X]$-module its corresponding quasi-coherent $\mathscr{O}_X$-module. For any $k[X]$-module $M$ and any $f \in k[X]$, we write $U_f$ for $\{ x \in X : f(x) \neq 0 \}$ and denote by $M_f$ the localization of $M$ with respect to $\{f^n : n \geq 0\}$. We recall that the natural map $M_f \to \mathrm{Sh}M(U_f)$ is an isomorphism. In particular, we always have $\mathrm{Sh} M (X) = M$.

For any $x \in X$, we denote by $e_x$ the character of $k[X]$ defined by the evaluation at $x$. If $M$ is any $k[X]$-module and $x \in X$, we write $M|_x= k \otimes_{e_x} M$ and the natural map $M \to M|_x, m \mapsto m|_x = 1 \otimes m$ is called the evaluation of $M$ at $x$. Note that the exact sequence of $k[X]$-modules
$$0 \to x M \to M \to M|_x \to 0$$
is functorial in $M$. In the following, we shall use the following lemma.

\begin{lemma} \label{free}
Let $M$ be a finitely generated $k[X]$-module. If the map $x \mapsto \mathrm{dim}_k M|_x$ is locally constant on $X$ then $\mathrm{Sh}M$ is locally free.
\end{lemma}

\begin{proof}
    Assume that $x \mapsto \mathrm{dim}_k M|_x$ is locally constant. We fix $y \in X$ and consider a family $(m_i)_i$ of elements of $M$ such that $(m_i|_y)_i$ is a basis of $M|_y$. Since $M$ is finitely generated, there exists an open subset $U \subset X$ such that $(m_{i|U})_i$ generates $\mathrm{Sh}M_{|U}$, see \cite{Ducros}*{(3.3.15)}. Restricting $U$ if necessary, we can assume that $x \mapsto \mathrm{dim}_k M|_x$ is constant on $U$. Then $(m_i|_x)_i$ is a basis of $M|_x$ for all $x \in U$. For any open subset $V$ of $U$ and any family $(f_i)_i$ of elements of $k[V]$, if $\sum_i f_i (m_{i|V}) = 0$ then for every $x \in V$, we have $\sum_i f_i(x) (m_i|_x) = 0$, so $f_i= 0$ for every $i$. This proves that $(m_{i|U})_i$ is a basis of $\mathrm{Sh}M_{|U}$. We conclude that $\mathrm{Sh}M$ is locally free.
\end{proof}

\subsection{Actions of algebraic groups} \label{quotients}

Let $G$ be an algebraic group and let $X$ be a $G$-variety. A quotient of $X$ by $G$ in the sense of \cite{Borel}*{II.6.3} will be called \textit{geometric} and denoted by $G\backslash X$. It is unique up to isomorphism. Let $x \in X$ and let $G^x$ be its stabilizer. The orbit morphism $g \mapsto gx$ of $G$ onto the subvariety $Gx \subset X$ is a quotient of $G$ by $G^x$ (for the action by right translation) if and only if it is separable \cite{Borel}*{II.6.7}, which is automatic if $k$ has characteristic zero. We recall that $X$ always contains a closed orbit \cite{Borel}*{I.1.8}. Let us also mention that if $G$ is affine and $H$ is a closed subgroup of $G$ then $G/H$ always exists \cite{Borel}*{II.6.8}.

From now, let us assume that $G$ and $X$ are affine. When the algebra $k[X]^G$ of $G$-invariant regular functions on $X$ is finitely generated, its maximal spectrum defines an affine variety that we denote by $G \backslash \!\! \backslash X$ and call the \textit{affine categorical quotient} of $X$ by $G$. By construction, if the geometric quotient $G\backslash X$ exists and is affine, then $G\backslash \!\! \backslash X$ is well defined and coincides with $G \backslash X$.

Assume that $G$ is reductive. Then $k[X]^G$ is always finitely generated, so that $G \backslash \!\! \backslash X$ is always well defined. Moreover, the morphism $X \to G\dquo X$ induced by the inclusion $k[X]^G \subset k[X]$ is surjective and each of its fibers contains a unique closed orbit. If each orbit is closed (in particular if $G$ is finite) then $G\dquo X$ coincides with the geometric quotient of $X$ by $G$. We refer to \cite{Newstead}*{Chapter 3, §3} for more information.

\subsection{Induction for representations of affine groups}

The representations of algebraic groups considered in this paper are always assumed to be the sum of their (finite dimensional) rational subrepresentations. They will generally be infinite dimensional. Let $G$ be an affine group. By definition, the group law of $G$ induces a Hopf algebra structure on $k[G]$, see \cite{Borel}*{Chapter I, §1.5}. The category $\mathrm{Rep}(G)$ of representations of $G$ is then equivalent to the category of $k[G]$-comodules.

Let us fix a closed subgroup $H$ of $G$. If $V$ is any representation of $H$, let us denote by $(k[G]\otimes_k V)^H$ the space of $H$-invariant vectors of the tensor product representation $k[G]\otimes_k V$, where $k[G]$ is considered as a representation of $H$ for right translation. The action of $G$ on $k[G]$ by left translation induces a representation of $G$ on $(k[G]\otimes_k V)^H$. The latter is denoted by $I^G(V)$ and is called \textit{the representation of $G$ induced from $V$.}

The functor $I^G : \mathrm{Rep}(H) \to \mathrm{Rep}(G)$ is the right adjoint of the restriction functor $R_H :\mathrm{Rep}(G) \to \mathrm{Rep}(H)$. To be precise, if $W$ is a representation of $G$ and $\delta_W : W \to k[G] \otimes_k W$ denotes the associated comodule map, the natural bijection $\mathrm{Hom}_H(R_H(W),V) \to \mathrm{Hom}_G(W, I^G(V))$ is defined by $\phi \mapsto (\mathrm{id} \otimes \phi)\circ \delta_W$ and its inverse is the postcomposition by $e_1 \otimes \mathrm{id}_V$. This natural bijection will be referred to as the \textit{universal property of induction}.

By construction, note also that $I^G$ commutes with direct limits provided that the initial directed system is made of injective morphisms. In other words, for any directed system $(V_i \hookrightarrow V_j)_{i\leq j}$ of injective morphisms in $\mathrm{Rep}(H)$, we have $$I^G\!\left(\varinjlim V_i\right) = \varinjlim I^G(V_i).$$
We say that \textit{$I^G$ commutes with strict direct limits}.

The question of whether the map $e_1 \otimes \mathrm{id}_V : I^G(V) \subset k[G] \otimes_k V \to V$ is surjective will be crucial for our purpose. To conclude this section, let us give a brief account on this problem. The subgroup $H$ is said to be \textit{observable} if the above map is surjective for every representation $V$ of $H$. By the universal property of induction and the fact that $I^G$ commutes with strict direct limit, $H$ is observable if and only if every finite dimensional representation of $H$ is a $H$-homomorphic image of a finite dimensional representation of $G$. By considering contragredient representations, we see that this is in turn equivalent to the property that every finite dimensional representation of $H$ can be extended to a finite dimensional representation of $G$. One of the main results of \cite{BBHM} is that $H$ is observable if and only if $G/H$ is quasi-affine, that is, an open subvariety of an affine variety. We will only need the following weaker statement.

\begin{theorem}\label{surjectivitye1}
    If $G/H$ is affine then $H$ is observable.
\end{theorem}

\begin{proof}[Proof, based on \cite{BBHM}]
    Let us call extensible any (algebraic) character of $H$ that can be extended into a representation of $G$. An elementary criterion for the observability of $H$, given by \cite{BBHM}*{Theorem 1}, is the following: if for any extensible character $\chi$ of $H$, the dual character $\chi^{-1}$ is also extensible, then $H$ is observable.
    
    Assume that $G/H$ is affine. This means that the algebra $k[G]^H$ of right $H$-invariant regular functions on $G$ is finitely generated and its maximal spectrum is $G/H$. Let $\chi$ be an extensible character of $H$. Let $\tau$ denote the representation of $G$ on $k[G]$ by right translation. Let $f$ be a regular function on $G$ such that $f(1) \neq 0$ and $\tau_h f = \chi(h) f$ for every $h \in H$ (consider an appropriate matrix coefficient of a representation of $G$ extending $\chi$). Since the quotient morphism $\pi : G \to G/H$ is open and the zero set of $f$ is right $H$-invariant, it follows that $\pi(f^{-1}(0))$ is closed in $G/H$ and do not contain $\pi(1)$. Therefore, there exists $\varphi \in k[G]^H$ vanihsing on the zero set of $f$ such that $\varphi(1) \neq 0$ (we use here that $G/H$ is affine). By Hilbert's Nullstellensatz, $\varphi$ belongs to the radical of the ideal of $k[G]$ generated by $f$, that is, there exists a non-zero $\psi \in k[G]$ and $n \in \mathbb{N}$ such that $\varphi^n = \psi f$. Next, we have $\tau_h \psi = \chi^{-1}(h)\psi$ for all $h \in H$, which proves that $\chi^{-1}$ is extensible, since it can be realized as a subrepresentation of $(k[G], \tau)$. Using the criterion of the first paragraph, we conclude that $H$ is observable.
\end{proof}

\section{Irreducible covariant representations}

Let $G$ be an affine group and let $X$ be an affine right $G$-variety. In order to avoid separability problems, we assume that $k$ has characteristic zero. In particular, if $x \in X$, then $Gx$ identifies with $G/G^x$ as a $G$-variety. The goal of this section is to classify the irreducible covariant representations of $(G,X)$. 

Let us first introduce some vocabulary. Let $M$ be a covariant representation of $(G,X)$. A \textit{subrepresentation} of $M$ is a $k[X]$-submodule of $M$ which is stable by $G$. We call $M$ \textit{irreducible} if the zero subspace is its only proper subrepresentation. We say that it is \textit{finitely generated} if, as a covariant representation, it is generated by a finite subset. This is the same as being finitely generated as a $k[X]$-module, since the representation of $G$ on $M$ is assumed to be rationally spanned. If $N$ is another covariant representation of $(G,X)$, a $k[X]$-linear map $M \to N$ is said to be a \textit{$(G,X)$-homomorphism} if it intertwines the actions of $G$. The category of covariant representations of $(G,X)$ is denoted by $\mathrm{Rep}(G,X)$.

The following lemma essentially reduces the classification problem to the case where $G$ acts transitively on $X$.

\begin{lemma}\label{annihilator}
Let $M$ be a covariant representation $(G, X)$ and let $I$ be its annihilator in $k[X]$. If $M$ is irreducible, then there exists a closed orbit $Y \subset X$ such that $I$ is the ideal of functions vanishing on $Y$.
\end{lemma}

\begin{proof} The irreducibility of $M$ implies that it is finitely generated, hence it is finitely generated as a $k[X]$-module. If $J$ is a $G$-invariant ideal of $k[X]$ containing $I$ properly, then $JM$ is a non-zero covariant subrepresentation of $M$, so $JM = M$ and by applying Nakayama's lemma, we get that $(1 + J) \cap I \neq \emptyset$, hence $J = k[X]$. This shows that $I$ is a maximal proper $G$-invariant ideal of $k[X]$. It follows that its zero locus $Z$ is a minimal closed $G$-invariant subset and that $I = \sqrt{I} = \{ f \in k[X] : f_{|Z} = 0 \}$. Since $Z$ contains at least one closed $G$-orbit, it is itself a closed orbit.
\end{proof}

\subsection{The classification} Now we need to construct enough covariant representations. For that, we use induction from stabilizers of points of $X$.

Let $x \in X$ and $V$ a representation of $G^x$. The morphism $\varphi_x : g \in G \mapsto gx \in X$ induces a $k[X]$-algebra structure on $k[G]$. Combined with the representation of $G$ by left translation, this induces a covariant representation of $(G,X)$ on $k[G] \otimes_k V$. Then $I^G(V)$ is a covariant subrepresentation of $k[G] \otimes_k V$, that we denote by $I_x(V)$.

Let us define an \textit{induction datum} as a pair $(x,V)$, with $x$ an element of $G$ generating a closed orbit and $V$ an irreducible representation of $G^x$. Two induction data $(x,V)$ and $(x',V')$ are said to be equivalent if there exists $g \in G$ such that $x' = gx$ and a linear isomorphism $T : V \to V'$ such that
    \begin{equation}\label{twist} T(h\cdot v) = (ghg^{-1})\cdot T(v) \qquad (v \in V, h \in G^{x}).\end{equation}
We can now classify the irreducible covariant representations of $(G,X)$ as follows.

\begin{theorem}\label{classification}
    For any induction datum $(x,V)$, the covariant representation $I_x(V)$ is irreducible. Conversely, any irreducible covariant representation is isomorphic to one of that form. Moreover, for a given pair of induction data $(x,V)$ and $(x',V')$, the covariant representations $I_x(V)$ and $I_{x'}(V')$ are isomorphic if and only if $(x,V)$ is equivalent to $(x',V')$.
\end{theorem}

We first show how to reduce the proof to the case where $X$ is transitive. The proof of the theorem in the latter case is postponed until the end of the section.

\begin{proof}[Proof of the theorem assuming that is is true in the transitive case.]
    If $Y$ is closed orbit of $X$, then the comorphism of the inclusion $Y\subset X$ induces on every covariant representation of $(G,Y)$ a structure of covariant representation of $(G,X)$. This establishes an equivalence between the category $\mathrm{Rep}(G,Y)$ and the full subcategory of $\mathrm{Rep}(G,X)$ consisting of covariant representations whose annihilator in $k[X]$ is the ideal of regular functions vanishing on~$Y$.
    
    Assume that the theorem is true for a transitive action of $G$. If $(x,V)$ is an induction datum, then $I_x(V)$ is irreducible as a covariant representation of $(G,Gx)$, hence irreducible as a covariant representation of $(G,X)$. The second statement of the theorem derives from Lemma \ref{annihilator}. For the last statement, just note that two isomorphic covariant representations of $(G,X)$ have the same annihilator in~$k[X]$.
\end{proof}

\subsection{The transitive case} From now, we assume that $G$ acts transitively on $X$. If $x \in X$ and $M$ is a covariant representation of $(G,X)$ then the action of $G^x$ on $M$ induces a representation of $G^x$ on $M|_x$. The classification theorem will be a consequence of the following result.

\begin{theorem}\label{equivalence}
    Assume that $G$ acts transitively on $X$. Let $x$ be a point of $X$ and let us denote by $\mathrm{ev}_x$ the functor $M \mapsto M|_x$, from $\mathrm{Rep}(G,X)$ to $\mathrm{Rep}(G^x)$. Then $\mathrm{ev}_x$ and $I_x : \mathrm{Rep}(G^x) \to \mathrm{Rep}(G,X)$ are inverse functors.
\end{theorem}

Although it is not stated exactly in these terms, this theorem already appears in \cite{2authors} with a proof relying on results of \cite{3authors}. We propose a new proof, more elementary, which focuses on the local structure of induced modules.

Let us also emphasize that the characteristic zero hypothesis on $k$ is only needed to ensure that the morphism $\varphi_x : g \in G \mapsto gx \in Gx$ is separable, hence a geometric quotient of $G$ by $G^x$ (see Section 1.2).

 We begin with a series of lemmas. Let us fix $x \in X$ and write $H = G^x$. In the following, $W$ denotes any finite dimensional representation of $H$. For simplicity, we denote by $\varphi_x^*$ both the inclusion $I_x(W) \hookrightarrow k[G]\otimes_k W$ and the comorphism $k[X] \hookrightarrow k[G]$ induced by $\varphi_x$.

\begin{lemma}
     The sheaf $U \mapsto (k[\varphi_x^{-1}(U)]\otimes_k W)^H$ is precisely the quasi-coherent $\mathscr{O}_X$-module associated to $I_x(W)$.
\end{lemma}

\begin{proof}
    It is enough to check that for any $f \in k[X]$, the $k[U_f]$-linear map \begin{equation}\label{linearmap} I_x(W)_f = \left\{(k[G] \otimes_k W)^H\right\}_{f}  \longrightarrow (k[\varphi_x^{-1}(U_{f})]\otimes_k W)^H,\end{equation} induced by $\varphi_x^*$ and restriction to $U_f$, is an isomorphism. The surjectivity of (\ref{linearmap}) is easily checked by using that $k[\varphi_x^{-1}(U_f)] = k[U_{\varphi_x^*f}] = k[G]_{f}$. Moreover, as a $k[U_f]$-module, $k[\varphi_x^{-1}(U_f)] \otimes_k W$ identifies with $(k[G] \otimes_k W)_{f}$. Since the localization is an exact functor, it follows that (\ref{linearmap}) is also injective. 
\end{proof}

For any $g \in G$, the evaluation map at $g$ of the free $k[G]$-module $k[G] \otimes_k W$ is $e_g \otimes \mathrm{id}_W : k[G] \otimes_k W \to W$. Accordingly, for any $g \in G$ and $w \in k[G] \otimes_k W $, we write $w |_g = (e_g \otimes \mathrm{id}_W)(w)$. 

\begin{lemma}
    Let $g \in G$. If $(\xi_i)_i$ is a collection of elements of $I_x(W)$ such that $(\varphi_x^*\xi_i|_g)_i$ is a basis of $W$, then there exists an open subset $U$ of $X$ such that
    \begin{enumerate}[label = (\roman*)]
        \item \label{1} $(\xi_{i | U})_i$ is a basis of $\{\mathrm{Sh}\, I_x(W)\}_{|U}$,
        \item \label{2} $(\varphi_x^*\xi_i|_{g'})_i$ is a basis of $W$ for every $g' \in \varphi_x^{-1}(U)$.
    \end{enumerate}
\end{lemma}
\begin{proof}
    Let $(b_j)_j$ be a $k[G]$-basis of $k[G]\otimes_k W$. Let us denote by $a$ the $k[G]$-valued matrix such that $\varphi_x^*\xi_i = \sum_j a_{ij}b_j$ for all $i$. Since $(\varphi_x^*\xi_i|_g)_i$ and $(b_j|_g)_j$ are both bases of $W$, they have the same cardinality and we can assume that $a$ is a square matrix.
    If $D = U_{\mathrm{det}(a)}$ then the matrix $a_{|D} = (a_{ij|D})_{i,j}$ is invertible, which means that $(\varphi_x^*\xi_{i|D})_i$ is a $k[D]$-basis of $k[D]\otimes_k W$. Now $D$ is right $H$-saturated, because of the $H$-equivariance of the $\xi_i$. Thus, condition \textit{\ref{2}} is satisfied for $U = \varphi_x^{-1}(D)$. For condition \textit{\ref{1}}, just notice that if $\psi \in I_x(W)$ and $(\varphi_x^*\psi)_{|D} = \sum_i f_i (\varphi_x^*\xi_i)_{|D}$ with $f_i \in k[D]$, then the $H$-equivariance of the $\xi_i$ implies that $f_i$ is right $H$-invariant.
\end{proof}

Thanks to Theorem \ref{surjectivitye1}, a family $(\xi_i)_i$ as in the above lemma always exists for $g=1$. By translating this family by the action of $G$, we immediately see that the same holds for every $g \in G$. Thus we obtain the following.

\begin{proposition}\label{ponctuel}
    The $\mathscr{O}_X$-module $\mathrm{Sh}\,I_x(W)$ is locally free. Moreover, for any $g \in G$, the composition
\begin{equation}
    I_x(W) \xrightarrow{\varphi_x^*} k[G] \otimes_k W \xrightarrow{|_{g}} W
\end{equation}
identifies with the evaluation of $I_x(W)$ at $gx$.
\end{proposition}


This already shows that $\mathrm{ev}_x \circ I_x$ is isomorphic to the identity functor when restricted to finite dimensional representations of $H$. The last lemma we need is the following.

\begin{lemma} \label{pointwise}
Let $M$ be a finitely generated representation of $(G,X)$. Then $\mathrm{Sh} M$ is locally free, the map $M \to \prod_{y \in X} M|_y$ defined by $m \mapsto (m|_y)_{y \in X}$ is injective and if $M_x = 0$ then $M = 0$.
\end{lemma}

\begin{proof}
    The action of any $g \in G$ on $M$ induces a $k$-linear isomorphism $M|_x \to M |_{gx}$. Since $G$ acts transitively on $X$, the dimension of $M|_x$ is therefore independent of~$x$. Then, Lemma \ref{free} implies that $\mathrm{Sh} M$ is locally free.
    
    If $m \in M$ is such that $m|_y = 0$ for all $y\in X$, then $m$ vanishes locally because of the local freeness of $\mathrm{Sh} M$, hence $m =0$. Thus $M \to \prod_{y \in X} M|_y$ is injective. In particular, since the dimension of $M|_y$ is independent of $y$, the condition $M|_x = 0$ implies that $\prod_{y \in X} M|_y = 0$ and $M = 0$.
\end{proof}

\begin{proof}[Proof of Theorem \ref{equivalence}]
    Let us first show that $\mathrm{ev}_x \circ I_x$ is isomorphic to the identity functor. Let $W$ be any representation of $G$. Let us write $W = \varinjlim W_i$, with $(W_i)_i$ the directed system of all finitely generated subrepresentations of $W$. Since the induction commutes with strict inductive limits, we have $I_x(W) = \varinjlim I_x(W_i)$. We showed above that $\mathrm{ev}_x \circ I_x$ is isomorphic to the identity functor when restricted to finite dimensional representations of $H$ (see Proposition \ref{ponctuel}). Since $\mathrm{ev}_x$ commutes with colimits, it follows that $I_x(W)|_x = \varinjlim I_x(W_i)|_x = \varinjlim W_i = W$.
    
    Now, let us prove that $I_x \circ \mathrm{ev}_x$ is isomorphic to the identity functor. Let $M$ be a representation of $(G,X)$. Given that $H$ fixes $x M$, there exists a unique representation of $H$ on $M|_x$ such that the evaluation map $M \mapsto M|_x$ is $H$-equivariant. By the universal property of induction, the latter induces a $G$-equivariant map $\iota_M : M \to I^G(M|_x) = I_x(M|_x)$. Using that $M \mapsto M|_x$ is $k[X]$-linear, with $M|_x$ considered as a $k[X]$ module through evaluation at $x$, and the $G$-equivariance, one can easily check that $\iota_M$ is also $k[X]$-linear. Our goal is to prove that $\iota_M$ is an isomorphism.
    
    Let us assume first that $M$ is finitely generated. Let $m$ be an element of $M$ such that $\iota_M(m) = 0$. Almost by definition, we then have $(gm)|_x = 0$, hence $m|_{g^{-1}x} = 0$, for all $g \in G$. From the injectivity statement of Lemma \ref{pointwise}, we deduce that $m = 0$. This proves the injectivity of $\iota_M$. On the other hand, the $k$-linear map $M|_x \to I_x(M|_x)|_x = M|_x$ induced by $\iota_M$ is just the identity, so
    $\{ I_x(M|_x) / M \}|_x = 0$
    by right exactness of $\mathrm{ev}_x$. The last statement of Lemma \ref{pointwise} then implies that $\iota_M$ is surjective.
    
    To conclude the proof, let us check that $\iota_M$ is an isomorphism for arbitrary $M$. We can express $M$ as the direct limit $\varinjlim M_i$ of its finitely generated subrepresentations $M_i$. Since $I_x$ and $\mathrm{ev}_x$ are inverse functors for finitely generated objects, $\mathrm{ev}_x$ preserves kernels of morphisms between finitely generated covariant representations of $(G,X)$. In particular, the directed system $(M_i|_x \to M_j|_x : i \leq j)$ is made of \textit{injective} morphisms. Since $I_x$ commutes with strict direct limits, we have
    $$I_x\left(\varinjlim M_i|_x\right) = \varinjlim I_x(M_i|_x).$$
    The latter identifies with $M = \varinjlim M_i$ through the isomorphisms $\iota_{M_i}$. Since $\mathrm{ev}_x$ commutes with colimits, we have $M|_x = \varinjlim M_i|_x$. In the end, we get an identification of $I_x(M|_x)$ with $M$ through $\iota_M$. \end{proof}
    
    \begin{proof}[Proof of Theorem \ref{classification} in the transitive case]
    
    The first two statements of the classification follow from Theorem \ref{equivalence}. It remains to check the last one.
    Let $(x,V)$ and $(x',V')$ be two induction data. Let us first assume that they are equivalent, so there exists $g \in G$ such that $x' = gx$ and a linear isomorphism $T : V\to V'$ satisfying (\ref{twist}). If $\rho$ denotes the action of $G$ on $k[G]$ by right translation, then the map
    $$\begin{array}{cclll}
        k[G] \otimes_k V & \longrightarrow & k[G] \otimes_kV' \\
        f \otimes v & \longmapsto & \rho_gf \otimes T(v) 
    \end{array}$$
    induces a $(G, X)$-isomorphism $I_x(V) \to I_{x'}(V')$.
    
    Conversely, let us assume that $I_{x'}(V')$ and $I_{x}(V)$ are isomorphic. From the foregoing, we can assume that $x = x'$ without loss of generality. Then $V$ is isomorphic to $V'$ by Theorem \ref{equivalence}.
\end{proof}

\section{Application to motion groups}

Our goal now is to apply the classification theorem of the previous section to the infinite-dimensional representation theory of \textit{motion groups}, that is, semi-direct products of Lie groups of the form $K \ltimes \mathfrak{p}$, with $\mathfrak{p}$ a real vector space and $K$ a compact Lie group acting linearly on $\mathfrak{p}$ by conjugation. For that purpose, we first recall the correspondence between compact Lie groups and complex reductive groups, as well as general aspects of the representation theory of motion groups.

\subsection{Compact Lie groups and complex reductive groups}
In this section and even more in the following one, we will frequently use implicitly the correspondence between compact Lie groups and reductive complex groups. Let us recall the main properties of that correspondence and refer to \cite{OV}*{Ch. 5} for more information. Let $G$ be a complex reductive group. A compact (with respect to the analytic topology) subgroup $U$ of $G$ is maximal among all compact subgroups of $G$ if and only if it is Zariski dense in $G$. We refer to such $U$ as a \textit{maximal compact subgroup} of $G$. If $\mathfrak{g}$ and $\mathfrak{u}$ denote the Lie algebras of $G$ and $U$ respectively, we have $\mathfrak{g} = \mathfrak{u} \oplus i \mathfrak{u}$. Moreover, the map $U \times \mathfrak{u} \to G, (u,X) \mapsto ue^{iX}$ is bijective ; we call it the \textit{Cartan decomposition} of $G$ with respect to $U$. There is a unique involutive anti-holomorphic automorphism $\tau$ of $G$ which fixes $U$, defined by
$$\tau(ue^{iX}) = ue^{-iX} \qquad (u \in U, X \in \mathfrak{u}).$$
The pair $(U,\tau)$ is said to be a \textit{compact real form} of $G$. If $K$ is closed subgroup of $U$, then its Zariski closure $\overline{K}$ is reductive and $\tau$-stable, with compact real form $(K,\tau_{|\overline{K}})$. Conversely, if $H$ is a $\tau$-stable Zariski closed subgroup of $G$, then it is reductive with compact real form $(H\cap U, \tau_{|H})$.

\subsection{General representation theory of motion groups} Let $G_0 = K\ltimes \mathfrak{p}$ be a motion group. We adopt the same terminology as in \cite{Warner}*{§4.2} and refer to it for more details on what follows.

A continuous representation $\pi$ of $G_0$ on a Banach space $E$ is said to be \textit{topologically completely irreducible} if the linear span of $\{\pi(g) : g \in G_0\}$ is dense in the space of bounded operator on $E$ with respect to the strong operator topology. A topologically completely irreducible representation of $G_0$ is necessarily topologically irreducible. For unitary representations on Hilbert spaces, both notions coincide. A continuous Banach representation of $G_0$ is said to be \textit{admissible} if its restriction to $K$ has finite multiplicity with respect to any irreducible representation of $K$, see for example \cite{Baldoni}*{Definition 8}. By a result of Godement \cite{Warner}*{Theorem 4.5.2.1}, a topologically completely irreducible representations of $G_0$ is necessarily admissible. In view of \cite{Warner}*{Proposition 4.2.1.5}, a continuous representation of $G_0$ on a Banach space is topologically completely irreducible if and only if it is admissible and topologically irreducible. For simplicity, such representations of $G_0$ will be called \textit{irreducible admissible}. Moreover, two irreducible admissible representations $(\pi, E_\pi)$ and $(\rho, E_\rho)$ of $G_0$ are said to be \textit{Naimark equivalent} if there exists an injective densely defined operator $T : E_\pi \to E_\rho$ whose graph is a closed subrepresentation of $(\pi\!\times\!\rho, E_\pi\! \times \!E_\rho)$. This is an equivalence relation, see \cite{Warner}*{Theorem 4.2.1.6}, which coincides with the unitary equivalence for irreducible unitary representations of $G_0$.

For many aspects, the study of admissible representations of $G_0$ reduces to that of $(\mathfrak{g}_0, K)$-modules (for a definition, see \cite{Baldoni}*{Definition 7}). The following standard facts may be found in \cite{Warner}*{§4.5.5}. Let $E$ be an admissible representation of $G_0$ and let us denote by $E_K$ the subspace of $E$ consisting of $K$-finite vectors. The elements of $E_K$ are smooth with respect to the action of $G_0$, therefore $E_K$ inherits a natural $(\mathfrak{g}_0, K)$-module structure. Moreover, $E$ is topologically irreducible if and only if $E_K$ is a simple $(\mathfrak{g}_0,K)$-module. In that case, if $E'$ is another irreducible admissible representation of $G_0$, then $E$ and $E'$ are Naimark equivalent if and only if $E_K$ and $E_K'$ are isomorphic as $(\mathfrak{g}_0, K)$-modules.

\subsection{Irreducible admissible representations of general motion groups} Let $K_\mathbb{C}$ be the complexification of $K$. It is an affine reductive group over $\mathbb{C}$. The action of $K$ on $\mathfrak{p}$ induces by complexification a representation of $K_\mathbb{C}$ on $\mathfrak{p}_\mathbb{C} = \mathfrak{p}\otimes_\mathbb{R}\mathbb{C}$. More generally, every continuous locally finite representation of $K$ on a complex vector space extends uniquely to a representation of $K_\mathbb{C}$ on the same vector space. Thus, any $(\mathfrak{g}_0,K)$-module $E$ inherits a representation of $K_\mathbb{C}$; moreover the action of $\mathfrak{p}$ induces on $E$ a structure of module over the symmetric algebra of $\mathfrak{p}_\mathbb{C}$. The latter identifies with the algebra of regular functions on the affine variety $\mathfrak{p}_\mathbb{C}^*$, hence any $(\mathfrak{g}_0,K)$-module can be endowed with a canonical covariant representation of $(K_\mathbb{C}, \mathfrak{p}^*_\mathbb{C})$, where the action of $K_\mathbb{C}$ on $\mathfrak{p}^*_\mathbb{C}$ is the dual to that on $\mathfrak{p}_\mathbb{C}$. This establishes a category equivalence between $\mathrm{Rep}(K_\mathbb{C}, \mathfrak{p}^*_\mathbb{C})$ and the category of $(\mathfrak{g}_0,K)$-modules. Theorem \ref{classification} therefore describes the equivalence classes of simple $(\mathfrak{g}_0,K)$-modules.

Let us now explain how we can deduce a classification of the irreducible admissible representations of $G_0$. We define a \textit{concrete datum} as a pair $(\lambda, \pi)$ with $\lambda$ an element of $\mathfrak{p}^*_\mathbb{C}$ which generates a closed $K_\mathbb{C}$-orbit and such that $K^\lambda$ is a maximal compact subgroup of $K_\mathbb{C}^\lambda$ (which is reductive by Matsushima's criterion), and $\pi$ an irreducible unitary representation of $K^\lambda$. Any concrete datum $(\lambda, \pi)$ defines an induction datum $(\lambda, \tilde{\pi})$ associated to $(K_\mathbb{C}, \mathfrak{p}^*_\mathbb{C})$, where $\tilde \pi$ is the unique extension of $\pi$ to a representation of $K_\mathbb{C}^\lambda$ on the same vector space. Let us call equivalent any two concrete data whose associated induction data are equivalent. 

Let us fix an real induction datum $(\lambda, \pi)$ and let us denote by $V$ the underlying vector space of $\pi$. Let $\mathcal{H}(\lambda,\pi)$ be the Hilbert space of square-integrable functions $f : K \to V$ such that
$$\pi(\gamma) f(k \gamma) = f(k) \qquad (k \in K, \gamma \in K^\lambda).$$
We endow it with the left regular action of $K$ and the linear action of $\mathfrak{p}$ defined by
$$(x \cdot f)(k) = e^{\langle k\lambda, x\rangle}f(k) \qquad (x \in \mathfrak{p}, f \in \mathcal{H}(\lambda,\pi), k \in K).$$
The combination of these two actions yields a continuous representation of $G_0$ on the Hilbert space $\mathcal{H}(\lambda,\pi)$.

\begin{theorem}\label{motiongroup}
For any concrete datum $(\lambda,\pi)$, the representation $\mathcal{H}(\lambda,\pi)$ is irreducible admissible. Moreover, any irreducible admissible representation of $G_0$ is of that form. Finally, two representation $\mathcal{H}(\lambda,\pi)$ and $\mathcal{H}(\lambda',\pi')$ are Naimark equivalent if and only if $(\lambda,\pi)$ and $(\lambda',\pi')$ are equivalent concrete data.
\end{theorem}

\begin{proof} 
    Let $(\lambda,\pi)$ be a concrete datum and let $V$ be the underlying vector space of $\pi$. The restriction of $\mathcal{H}(\lambda,\pi)$ to $K$ is a subrepresentation of $L^2(K) \otimes_\mathbb{C} V$, where $K$ acts trivially on the second factor. Hence, from the Peter-Weyl theorem, it follows that $\mathcal{H}(\lambda,\pi)$ is admissible. Let $(\lambda,\tilde\pi)$ be the induction datum associated to $(\lambda,\pi)$. The restriction to $K$ defines an isomorphism of $\mathbb{C}[K_\mathbb{C}]$ onto the algebra of continuous $K$-finite functions on $K$, which in turn induces, after tensoring by $V$, an linear embedding $I_\lambda(\tilde \pi) \to \mathcal{H}(\lambda,\pi)$. Using that $K^\lambda$ is a maximal compact subgroup of $K_\mathbb{C}^\lambda$, one can check that the image of this embedding is precisely the space $\mathcal{H}(\lambda,\pi)_K$ of $K$-finite vectors of $\mathcal{H}(\lambda,\pi)$. If we consider $I_\lambda(\tilde \pi)$ as a $(\mathfrak{g}_0,K)$-module through the category equivalence mentioned at the top of the subsection, then this embedding clearly defines an isomorphism of $(\mathfrak{g}_0,K)$-modules $I_\lambda(\tilde \pi) \to \mathcal{H}(\lambda,\pi)_K$. Since $I_\lambda(\tilde \pi)$ is simple, we conclude that $\mathcal{H}(\lambda,\pi)$ is topologically irreducible.
    
    Let $E$ be an irreducible admissible representation of $G_0$. The associated $(\mathfrak{g}_0,K)$-module $E_K$ of $K$-finite vectors is simple. It is thus equivalent to $I_\lambda(V)$ for some induction datum $(\lambda, V)$ of the pair $(K_\mathbb{C}, \mathfrak{p}^*_\mathbb{C})$. Given that every maximal compact subgroup of $K^\lambda_\mathbb{C}$ can be extended to a maximal compact subgroup of $K_\mathbb{C}$ and that the maximal compact subgroups of $K_\mathbb{C}$ are all conjugated, there exists $k \in K_\mathbb{C}$ such that $k^{-1} K k \cap K_\mathbb{C}^\lambda$ is a maximal compact subgroup of $K^\lambda_\mathbb{C}$. Then $K^{k\lambda}$ is a maximal compact subgroup of $K^{k\lambda}_\mathbb{C}$. Thus, by considering an equivalent induction data if necessary, we can assume without loss of generality that $K^\lambda$ is a maximal compact subgroup of $K_\mathbb{C}^\lambda$. If $\pi$ denotes the restriction of $V$ to $K$, then the $(\mathfrak{g}_0, K)$-module $\mathcal{H}(\lambda,\pi)_K$ is isomorphic to $I_\lambda(V)$. It follows that $\mathcal{H}(\lambda,\pi)$ and $E$ are Naimark equivalent.
    
    The last statement is also a straighforward consequence of the fact that two irreducible admissible representation of $G_0$ are Naimark equivalent if and only if their associated $(\mathfrak{g}_0,K)$-modules are isomorphic.
\end{proof}

\begin{remark}
    The above proof shows two non-trivial facts. First, the irreducible admissible representations of $G_0$ on Hilbert spaces exhaust its admissible dual. Second, every simple $(\mathfrak{g}_0,K)$-module can be integrated to an irreducible admissible representation of $G_0$.
\end{remark}

\subsection{The case of Cartan motion groups}
From now, let us focus on the case where $G_0$ is the motion group of a real reductive group $G$ with Cartan decomposition $Ke^{\mathfrak{p}}$. When $G$ is semisimple, the irreducible admissible representations of $G_0$ have already been classified explicitly \cites{DelormeCh, Rader}. Here we prove that the same classification holds for $G$ reductive in the sense of \cite{Knapp}*{Chapter {\rm VII}, §2}. Our approach is new and much simpler: it relies on the description of closed orbits, properties of stabilizers and Theorem \ref{motiongroup}.

Since $\mathfrak{p}_\mathbb{C}$ is self-dual as a representation of $K_\mathbb{C}$, we may equally consider $(\mathfrak{g}_0,K)$-modules as covariant representations of $(K_\mathbb{C}, \mathfrak{p}_\mathbb{C})$, where $K_\mathbb{C}$ acts on $\mathfrak{p}_\mathbb{C}$ through the adjoint representation. Let $\mathfrak{a} \subset \mathfrak{p}$ be a maximal abelian subspace. A version of Chevalley's restriction theorem \cite{Wallach}*{§3.1.2} asserts that the restriction morphism $\mathbb{C}[\mathfrak{p}_\mathbb{C}] \to \mathbb{C}[\mathfrak{a}_\mathbb{C}]$ induces a isomorphism of algebras
$$\mathbb{C}[\mathfrak{p}_\mathbb{C}]^{K_\mathbb{C}} \longrightarrow \mathbb{C}[\mathfrak{a}_\mathbb{C}]^W,$$
where $W$ denotes the restricted Weyl group related to $\mathfrak{a}$.
In other words, we have an isomorphism of affine varieties $W \backslash \mathfrak{a}_\mathbb{C} \to K_\mathbb{C} \dquo \mathfrak{p}_\mathbb{C}$. On the other hand, from \cite{KostantRallis}*{Theorem 1, Proposition 16}, every element of $\mathfrak{a}_\mathbb{C}$ generates a closed $K_\mathbb{C}$-orbit. Since every fiber of the canonical morphism $\mathfrak{p}_\mathbb{C} \to K_\mathbb{C} \dquo \mathfrak{p}_\mathbb{C}$ contains a unique closed orbit, we get the following.

\begin{proposition}
    The closed $K_\mathbb{C}$-orbits of $\mathfrak{p}_\mathbb{C}$ are exactly the orbits intersecting~$\mathfrak{a}_\mathbb{C}$. Moreover, two elements of $\mathfrak{a}_\mathbb{C}$ are in the same $K_\mathbb{C}$-orbit if and only if they are related by an element of $W$.
\end{proposition}

Now, regarding the stabilizers, the following holds.

\begin{proposition} \label{complexstab}
If $x \in \mathfrak{a}_\mathbb{C}$ then $K^x$ is a maximal compact subgroup of $K_\mathbb{C}^x$.
\end{proposition}

To give the proof, we need first to introduce some notation and vocabulary. A subgroup of a topological group $H$ is said to be \textit{full} if it meets every connected component of $H$. We denote by $\theta$ the complexified Cartan involution on $\mathfrak{g}_\mathbb{C}$ and write $\tau$ for the unique antilinear involutive automorphism of $\mathfrak{g}_\mathbb{C}$ such that $\tau_{|\mathfrak{g}} = \mathrm{id}_\mathfrak{g}$. If $H$ is any subgroup of the adjoint group $\mathrm{Int}(\mathfrak{g}_\mathbb{C})$ of $\mathfrak{g}_\mathbb{C}$ and $\sigma$ is any element of $\mathrm{GL}(\mathfrak{g}_\mathbb{C})$, then we write $H^\sigma = \{ h \in H : \sigma h = h \sigma\}$. We put $L = \mathrm{Int}(\mathfrak{g}_\mathbb{C})^\theta$, $U = \mathrm{Int}(\mathfrak{g}_\mathbb{C})^{\tau\theta}$ and $F = \{ a \in \exp[\mathrm{ad}_{\mathfrak{g}_\mathbb{C}} (\mathfrak{a}_\mathbb{C})] : a^2 = 1\}$. Since $G$ is of inner type, the adjoint action defines an open homomorphism of complex groups $\delta : K_\mathbb{C} \to L$. Finally, note that $\delta$ maps $K$ to $U$, the latter being a maximal compact subgroup of $\mathrm{Int}(\mathfrak{g}_\mathbb{C})$. This means that $\delta$ preserves Cartan decompositions: if $\sigma_K$ and $\sigma_U$ denote the anti-holomorphic involutive automorphisms of $K_\mathbb{C}$ and $\mathrm{Int}(\mathfrak{g}_\mathbb{C})$ associated to $K$ and $U$ respectively, then we have $\delta \circ \sigma_K = \sigma_U \circ \delta$.

\begin{proof}
An element $y \in \mathfrak{k}_\mathbb{C}$ commutes with $x$ if and only if $\tau(y)$ do so. One can easily prove this by expanding $y$ in root vectors, see the beginning of \cite{KostantRallis}*{§I.6}. It follows that the Lie algebra of $K_\mathbb{C}^x$ is the complexification of that of $K^x$. To prove the proposition, it is thus enough to show that $K^x$ is a full subgroup of $K_\mathbb{C}^x$.

To begin with, let us prove that $F$ is a full subgroup of the centralizer $L^x$ of $x$ in $L$. Let $\mathfrak{g}_\mathbb{C}^x$ be the centralizer of $x$ in $\mathfrak{g}_\mathbb{C}$ and let $Z_x$ be the centralizer of $\mathfrak{g}_\mathbb{C}^x$ in $L^x$. The proof of \cite{KostantRallis}*{lemma 10} shows that $F Z_x$ is a full subgroup of $L^x$. Moreover, $Z_x$ is contained in the centralizer $M$ of $a_\mathbb{C}$ in $L$, which contains $F$ as a full subgroup by \cite{KostantRallis}*{lemma 20}. Since $M \subset L^x$, it follows that $F$ is a full subgroup of $L^x$.

Let us now conclude. Using that $\delta : K_\mathbb{C} \to \mathrm{Int}(\mathfrak{g}_\mathbb{C})$ preserves Cartan decompositions, we obtain $\delta(K_\mathbb{C})^\tau = \delta(K)$ and $\ker \delta = K \exp(i \mathfrak{k}\cap\ker \delta_*)$, where $\delta_*$ is the Lie algebra morphism induced by $\delta$. These two facts imply respectively that $\delta(K_\mathbb{C}^x)^\tau = \delta(K^x)$ and that $K \cap \ker \delta$ is a full subgroup of $\ker \delta$. From that fact that $\delta(K_\mathbb{C}^x)^\tau$ contains $F$, we get that $\delta(K^x)$ is a full subgroup of $\delta(K_\mathbb{C}^x)$. Since $K^x$ contains $K\cap \ker \delta$, it follows that $K^x$ is a full subgroup of $K_\mathbb{C}^x$.
\end{proof}

Let $\mathfrak{q}$ be the orthogonal complement of $\mathfrak{a}$ with respect to some $K$-invariant inner product on $\mathfrak{p}$. Let us identify $\mathfrak{a}_\mathbb{C}^*$ with the space of linear forms on $\mathfrak{p}_\mathbb{C}$ vanishing on $\mathfrak{q}_\mathbb{C}$. The above proposition shows that the stabilizer in $K_\mathbb{C}$ of an element of $\mathfrak{a}_\mathbb{C}^*$, with respect to the coadjoint action, coincides with the complexification of its stabilizer in $K$. Thus, for every $\lambda \in \mathfrak{a}^*_\mathbb{C}$ and every irreducible unitary representation $\pi$ of $K^\lambda$, the pair $(\lambda, \pi)$ is a concrete datum and one can consider the irreducible admissible representation $\mathcal{H}(\lambda, \pi)$ of $G_0$. Then, from all the foregoing analysis and Theorem \ref{motiongroup}, we obtain the following classification.

\begin{theorem}
    Every topologically completely irreducible representation of $G_0$ is Naimark equivalent to one of the form $\mathcal{H}(\lambda, \pi)$, for some $\lambda \in \mathfrak{a}^*_\mathbb{C}$ and irreducible unitary representation $\pi$ of $K^\lambda$. Two such representations $\mathcal{H}(\lambda, \pi)$ and $\mathcal{H}(\lambda', \pi')$ are Naimark equivalent if and only if $(\lambda, \pi)$ and $(\lambda', \pi')$ are related by an element of the Weyl group $W$.
\end{theorem}

\section{Quantum analogues of Cartan motion groups and their irreducible representations}

Real semisimple quantum groups have been introduced by De Commer \cite{DCquantisation} (see also \cite{YGE}*{§1} for a brief overview). They form a generalization of complex semisimple quantum groups (see the monograph \cite{VoigtYuncken}). Just as a Hecke algebra can be defined from the symmetric pair corresponding to a semisimple Lie group, one can associate to each real semisimple quantum group an approximately unital algebra. The non-degenerate modules over this algebra are viewed as analogues of Harish Chandra's $(\mathfrak{g},K)$-modules. In the classical case, the isomorphism classes of simple $(\mathfrak{g},K)$-modules and the isomorphism classes of simple $(\mathfrak{g}_0,K)$-modules, where $(\mathfrak{g}_0,K)$ is the pair of the Cartan motion group, share a natural common parametrization known as the Mackey bijection \cites{Afgoustidis1}. As suggested by a small example \cite{YGE}, a similar correspondence may hold for real semisimple quantum groups, provided that the pair\footnote{Or rather the affine action pair $(K_\mathbb{C}, \mathfrak{p}_\mathbb{C}^*$), see the previous section.} $(\mathfrak{g}_0,K)$ is replaced by an affine action pair of the form $(G^\theta, G /G^\theta)$, where $G$ is a simply connected complex semisimple group and $\theta$ an involution of $G$. Motivated by these considerations, the goal of this section is to give a simple description of the irreducible covariant representations of such an action pair. This could give some insights into the representation theory of real semisimple quantum groups, which is completely unknown in general.

The material of Subsection 3.1 will be used throughout this section. For the rest of the section, we fix a simply connected semisimple complex algebraic group $G$ and an involutive automorphism $\theta$ of $G$. We also consider a compact real form $(U,\tau)$ of $G$ such that $\tau$ commutes with $\theta$ (such a compact real form always exists: apply \cite{Knapp}*{Theorem 6.16} to the Lie algebra of $G$ viewed as a real one). Let $H$ denote the subgroup of elements of $G$ fixed by $\theta$. It is clear that $H$ is Zariski closed and stable by $\tau$, so it is reductive. Since $U$ is simply connected, $K = U\cap H$ is connected \cite{Bourbaki}*{p.48, Théorème 1, c)}, hence so is $H$. We are interested in the representation theory of the affine action pair $(H, G/H)$. Theorem \ref{classification} already gives an abstract classification of the irreducible representations, but our goal here is to make it explicit by giving a simple parametrization of the closed $H$-orbits of $G/H$.

\subsection{Background from the theory of symmetric spaces of compact type} To achieve the above goal, we start by re-expressing some results about the symmetric space $U/K$, for which we refer to \cite{Helgason1}*{VII} and \cite{Helgason2}*{V.4}.

Let $A \subset G$ be a $\tau$-stable torus which is maximal among those  contained in
$$\{ g \in G : \theta(g) = g^{-1}\}.$$ Let us fix a maximal torus $T$ of $G$ containing $A$, which is stable by $\theta$ and $\tau$. Denoting by $F$ the subgroup of order two elements of $A$, note that $$F = A \cap H \subset K.$$ Let $(X^*(T), \Phi, X_*(T), \Phi^\vee)$ be the root datum associated to the pair $(G,T)$. Since $G$ is simply connected, the coweight lattice $X_*(T)$ is $\mathbb{Z}$-spanned by the set of coroots $\Phi^\vee$. The set of restricted roots $\Phi_{|A} \setminus \{0\}$ is denoted by $\Sigma$. We recall that $(\mathbb{R}\otimes_\mathbb{Z} X^*(A), \Sigma)$ is a root system and that its Weyl group $W$ identifies with $N_K(A)/Z_K(A)$, the quotient of the normalizer of $A$ in $K$ by its centralizer in $K$, through the action of $N_K(A)$ on $A$ by conjugation. Let $(\mathbb{R}\otimes_\mathbb{Z}X_*(A), \Sigma^\vee)$ be the dual root system of $\Sigma$. From \cite{Helgason1}*{Theorem 8.5}, it follows that
$X_*(A) = \mathrm{span}_\mathbb{Z} (\Sigma^\vee),$
in particular $(X^*(A), \Sigma, X_*(A), \Sigma^\vee)$ forms a root datum. In the following, we fix positive root systems $\Phi^+$ and $\Sigma^+$ of $\Phi$ and $\Sigma$ respectively such that $(\Phi^+)_{|A} \subset \{0\} \cup \Sigma^+$. We equip $X^*(A)$ with the order defined by
$$\lambda \leq \lambda' \Longleftrightarrow \forall\alpha \in \Sigma^+,\; \langle \lambda, \alpha^\vee\rangle \leq \langle \lambda', \alpha^\vee\rangle,$$
where for any $\alpha \in \Sigma^+$, we denote by $\alpha^\vee$ the associated coroot. The set of dominant weights, denoted by $X^*(A)^+$, corresponds to the set of non-negative weights with respect to this ordering.

A representation $V$ of $G$ is called spherical if its subspace of $H$-fixed vectors $V^H$ is non-zero. In that case, every non-zero element of $V^H$ is called a spherical vector. The characterization of the irreducible spherical representations of $G$ is due to Helgason \cite{Helgason2}*{V, Corollary 4.2}. Let $T_H$ be a maximal $\tau$-stable torus of $H$ which is contained in $T$. An irreducible representation of $G$ with highest weight $\lambda$ is spherical if and only if $\lambda$ is trivial on $T_H$ and $\langle \lambda, \alpha^\vee \rangle \in 2\mathbb{Z}$ for all $\alpha \in \Sigma$. The last condition amounts to $\lambda_{|A} \in 2X^*(A)$, or equivalently, $\lambda$ is trivial on $F$. Since $T\cap H = T_H F$, we get the following statement.

\begin{proposition}[Helgason] \label{characspherical}
    An irreducible representation of $G$ is spherical if and only if its highest weight is trivial on $T \cap H$.
\end{proposition}

The inclusion $A \subset T$ induces an isomorphism $A/F \to T/H\cap T$, hence every weight $\lambda \in 2X^*(A)$ has a unique extension $\lambda^T \in X^*(T)$ which is trivial on $T\cap H$. Finally, about the spherical vectors of irreducible spherical representations of $G$, the following is known.

\begin{proposition}\label{spherical}
    Let $\lambda \in 2X^*(A)^+$ and let $V$ be an irreducible spherical representation of $G$ with highest weight $\lambda^T$. Let $\Pi(V)$ denote the set of $A$-weights of $V$. The following holds:
    \begin{itemize}
        \item $\lambda$ is the highest $A$-weight of $V$ and the associated eigenspace is one-dimensional,
        \item $V^H$ is one-dimensional,
        \item if $v = \sum_{ \omega \in \Pi(V)}v_\omega$ is the expansion in $A$-weight vectors of a spherical vector $v$ of $V$, then $v_{\lambda} \neq 0$.
    \end{itemize}
\end{proposition}

For the proof, we refer to the beginning of \cite{KKS}*{§4.1}.

\subsection{A restriction theorem and a parametrization of the closed orbits}

Our first step towards the parametrization of closed $H$-orbits of $G/H$ is to prove an integrated version of the Chevalley restriction theorem. Let us equip $G$ with the algebraic action of $H \times H$ by left and right multiplication. The natural action of $W$ on $A$ preserves $F$, hence $A/F$ inherits an action of $W$.

\begin{theorem} \label{restriction}
    The inclusion $A \subset G$ induces an isomorphism of affine varieties $$W \backslash A /F \to (H \times H) \dquo G.$$
\end{theorem}

\begin{proof}
    Our goal is to prove that the restriction to $A$ induces an isomorphism $\mathbb{C}[G]^{H\times H} \to \mathbb{C}[A]^{W\times F}$.
    Since $F$ is contained in $H$ and $W$ acts on $A$ by conjugation by elements of $K \subset H$, the latter morphism is well-defined. It is injective: by \cite{Helgason1}*{Theorem 6.7}, we have $U \subset KAK$, so $HAH$ is Zariski dense in $G$. It remains to prove the surjectivity.
    
    If a weight $\lambda$ of $A$ is considered as a regular function on $A$, rather than as an algebraic group morphism $A \to \mathbb{C}^*$, let us denote it by $e^\lambda$. In that way, $\lambda \mapsto e^\lambda$ defines a group morphism $X^*(A) \to \mathbb{C}[A]^\times$. We recall that $(e^\lambda)_{\lambda \in X^*(A)}$ is a basis of $\mathbb{C}[A]$. The algebra of $F$-invariant regular functions on $A$ is the image of the comorphism of $a \in A \mapsto a^2 \in A$. The latter maps $e^\lambda$ to $e^{2\lambda}$ for any $\lambda \in X^*(A)$, so $(e^\lambda)_{\lambda \in 2X^*(A)}$ is a basis of $\mathbb{C}[A]^F$.
    Since every $W$-orbit of $X^*(A)$ intersects $X^*(A)^+$, we deduce from the above that the elements $$\langle e^\lambda \rangle_W =|W|^{-1} \sum_{w \in W} e^{w(\lambda)},$$ as $\lambda$ ranges over $2X^*(A)^+$, span $\mathbb{C}[A]^{W\times F}$. Therefore, to conclude the proof, it is enough to show that for every $\lambda \in 2X^*(A)^+$, there exists a $H$-bi-invariant regular function on $G$ which restricts to $\langle e^\lambda \rangle_W$ on $A$. We prove this by induction.
    
    Let $\lambda \in 2X^*(A)^+$. Let us assume that for every $\omega \in 2X^*(A)^+$ such that $\omega\! <\! \lambda$, we can lift $\langle e^\omega\rangle_{W}$ into a $H$-bi-invariant regular function on $G$. Let $V$ be an irreducible representation of $G$ with highest weight $\lambda^T$ (see below Proposition \ref{characspherical}). Both $V$ and its contragredient representation $V^*$ are spherical. Let $\phi \in V^*$ and $v \in V$ be spherical vectors. We write $v = \sum_\omega v_\omega$ and $\phi = \sum_\omega \phi_\omega$ for their expansion in $A$-weight vectors. Let $c_{\phi,v}$ be the regular $H$-bi-invariant regular function on $G$ defined by
    $$c_{\phi,v}(g) = \langle \phi, gv\rangle \qquad (g \in G).$$
    Then we have $(c_{\phi,v})_{|A} = \sum_\omega \langle \phi_{-\omega}, v_\omega\rangle e^\omega$, where the sum ranges over the $A$-weights of $V$. Since $(c_{\phi,v})_{|A}$ is invariant by $F$ and $W$, we have $\langle \phi_{-\omega}, v_\omega\rangle = 0$ whenever $w \notin 2X^*(A)$ and $\langle \phi_{-\omega}, v_\omega\rangle = \langle \phi_{-\omega'}, v_{\omega'}\rangle$ whenever $\omega' \in W \omega$. Thus, there exist a set $S$ of $A$-weights of $V$, which contains $\lambda$ and is included in $2X^*(A)^+$, such that $(c_{\phi,v})_{|A} = \sum_{\sigma\in S} \langle \phi_{-\sigma}, v_\sigma\rangle \langle e^\sigma \rangle_W$. Since $\sigma \leq \lambda$ for every $\sigma \in S$, the induction hypothesis implies that $\sum_{\sigma \neq \lambda} \langle \phi_{-\sigma}, v_\sigma\rangle \langle e^\sigma \rangle_W$ can be lifted into a $H\times H$-invariant regular function on $G$. By Proposition \ref{spherical}, we have $\langle \phi_{-\lambda},v_\lambda\rangle \neq 0$, hence $\langle e^\lambda\rangle_W$ can also be lifted.
\end{proof}

Let $\pi$ denote the quotient morphism $G \to G/H$. Through the latter, we identify $\pi(A) \subset G/H$ with $A/F$. The closed orbits of $G/H$ under the action of $H$ are parametrized as follows.

\begin{proposition}
Every element of $A/F$ generates a closed $H$-orbit of $G/H$. Conversely, every closed $H$-orbit of $G/H$ intersects $A/F$. If $a$ and $a'$ are two elements of $A/F$, then $a$ belongs the $H$-orbit of $a'$ if and only if $a \in W a'$.
\end{proposition}
\begin{proof}
    
    First, let us check that every closed orbit of $G/H$ intersects $A/F$. This will follow from work of Matsuki--Oshima \cite{MatsukiOshima}. Let $\varphi: G \to G$ be the map defined by
    $$\varphi(g) = g \theta(g)^{-1} \qquad (g \in G).$$
    Matsuki and Oshima proved that $\varphi(G)$ is a closed embedded submanifold of $G$ and that it induces a diffeomorphism $\tilde \varphi : G/H \to \varphi(G)$. If we equip $\varphi(G)$ with the action of $H$ by conjugation, then $\tilde \varphi$ is $H$-equivariant. From \cite{MatsukiOshima}*{Proposition 2, (ii)}, it follows that every analytically closed orbit of $\varphi(G)$ contains a semi-simple element. By \cite{MatsukiOshima}*{Corollary (ii) p.412}, every semi-simple element in $\varphi(G)$ is $H$-conjugated to an element of the centralizer $Z_{\varphi(G)}(A)$ of $A$ in $\varphi(G)$. Now, by \cite{MatsukiOshima}*{Lemma 8}, we have $A = Z_{\varphi(G)}(A)$. To sum up, if $O$ is a closed $H$-orbit of $G/H$, then $\tilde \varphi(O)$ is an analytically closed orbit of $\varphi(G)$, hence $A \cap \tilde \varphi(O) \neq \emptyset$. Since $A = \varphi(A) = \tilde\varphi(A/F)$, it follows that $O$ intersects $A/F$.
    
    Now let us prove that any element of $A/F = \pi(A)$ generates a closed orbit. Let us fix $a \in A$ and let us write $S_a$ for the fiber of the morphism $G \to (H\times H) \dquo G$ which contains $a$. As explained in Subsection \ref{quotients}, there exists a unique closed $H\times H$-orbit $O$ in $S_a$. Since $\pi$ is open and $O$ is right $H$-saturated, $\pi(O)$ is a closed $H$-orbit of $G/H$. By the above, we thus have $\pi(O)\cap \pi(A) \neq \emptyset$, so that $O \cap A \neq \emptyset$. Now, any element of $O\cap A$ has the same image as $a$ in $(H\times H) \dquo G$, hence in $W\backslash A /F$ by Theorem \ref{restriction}. This proves that $W(Fa)$ intersects $O$. Since $F \subset H$ and $W$ acts on $A$ by conjugation by elements of $H$, we have $a \in O$, which proves that $\pi(a)$ generates the closed orbit $\pi(O)$.
    
    Two elements $a$ and $a'$ of $A/F$ are in same $H$-orbit if and only if their preimages by $\pi$ are all in the same (closed) $H\times H$-orbit. Since every fiber of the morphism $G \to (H\times H)\dquo G$ contains a unique closed orbit, by Theorem \ref{restriction} we get that $a$ and $a'$ are in the same $H$-orbit if and only if they belong to the same $W$-orbit.
\end{proof}

In conclusion, from Theorem \ref{classification}, we get the following description of the irreducible covariant representation of $(H, G/H)$.
\begin{theorem}
    Any irreducible covariant representation of $(H, G/H)$ is isomorphic to $I_{a}(V)$ for some $a \in A/F$ and irreducible representation $V$ of $H^{a}$. Two such irreducible representations $I_{a}(V)$ and $I_{a'}(V')$ are isomorphic if and only if $(a,V)$ and $(a',V')$ are related by an element of $W$.
\end{theorem}


\begin{thebibliography}{99}
    \bib{Afgoustidis1}{article}{
        author = {Afgoustidis, A.},
        title = {On the analogy between real reductive groups and Cartan motion groups: The Mackey-Higson bijection},
        journal = {Cambridge Journal of Mathematics},
        volume = {9, no.3},
        date = {2021},
        pages = {551--575}
    }
    \bib{Baldoni}{article}{
        title = {General representation theory of real reductive Lie groups},
        author = {Baldoni, M. W.},
        pages = {61--72},
        book = {
            title = {Representation theory and automorphic forms (Edinburgh, 1996)},
            series = {Proc. Sympos. Pure Math.},
            volume = {61},
            publisher = {American Mathematical Society},
            address = {Providence, RI},
            date = {1997}
        }
    }
    \bib{BBHM}{article}{
        title = {Extensions of Representations of Algebraic Linear Groups},
        author = {Bialynicki-Birula, A.},
        author = {Hochschild, G.},
        author = {Mostow, G. D.},
        journal = {American Journal of Mathematics},
        volume = {85, no. 1},
        date = {1963},
        pages = {131-144}
        }
    \bib{Bourbaki}{book}{
        title = {Groupes et algèbres de Lie},
        subtitle = {Chapitre 9 Groupes de Lie réels compacts},
        author = {Bourbaki, N.},
        publisher = {Springer Berlin, Heidelberg},
        date = {2007}
        }
        
    \bib{Borel}{book}{
        title = {Linear Algebraic Groups},
        author = {Borel, A.},
        publisher = {Springer New York, NY},
        date = {1991}
        }
    \bib{3authors}{article}{
        title = {Induced Modules and Affine Quotients},
        author = {Cline, E.},
        author = {Parshall, B.},
        author = {Scott, L.},
        journal = {Mathematische Annalen},
        volume = {230},
        date = {1977},
        pages = {1-14}
        }
    \bib{DelormeCh}{article}{
        title = {Sur les représentations des groupes de déplacements de Cartan},
        author = {Champetier, C.},
        author = {Delorme, P.},
        journal = {Journal of Functional Analysis},
        volume = {43},
        date = {1981},
        pages = {258-279}
        }
    \bib{DCquantisation}{article}{
        title = {Quantisation of semisimple real Lie groups},
        author = {De Commer, K.},
        journal = {Archivum Mathematicum},
        volume = {060, no.5},
        date = {2024},
        pages = {285--310}
    }
    
    \bib{Ducros}{article}{
        title = {Introduction à la théorie des schémas},
        author = {Ducros, A.},
        note = {arXiv:1401.0959},
        }
    
    \bib{YGE}{article}{
        title = {The smallest quantum Mackey analogy},
        author = {Gaudillot--Estrada, Y.},
        note = {arXiv:2602.23222}
    }
    \bib{Helgason2}{book}{
        title = {Groups and Geometric Analysis: Integral Geometry, Invariant Differential Operators, and Spherical Functions},
        author = {Helgason, S.},
        series = {Mathematical Surveys and Monographs},
        volume = {83},
        publisher = {American Mathematical Society},
        date = {2000}}
    \bib{Helgason1}{book}{
        title = {Differential Geometry, Lie Groups, and Symmetric Spaces},
        author = {Helgason, S.},
        publisher = {American Mathematical Society},
        series = {Graduate Studies in Mathematics},
        volume = {34},
        date = {2001}}
    \bib{MAhigson}{article}{
        title = {The Mackey analogy and $K$-theory},
        author = {Higson, N.},
        pages = {149--172},
        book = {
            title = {Group representations, ergodic theory, and mathematical physics: a tribute to George W. Mackey},
            series = {Contemporary Mathematics},
            volume = {449},
            publisher = {American Mathematical Society},
            address = {Providence, RI},
            date = {2008}
        }
    }
    \bib{Knapp}{book}{
        title = {Lie Groups Beyond an Introduction},
        author = {Knapp, A.},
        publisher = {Birkhäuser Boston, MA},
        date = {1996}}
        
    \bib{KostantRallis}{article}{
        title = {Orbits and Representations Associated with Symmetric Spaces},
        author = {Kostant, B.},
        author = {Rallis, S.},
        journal = {American Journal of Mathematics},
        volume = {93, no. 3},
        pages = {pp. 753-809},
        date = {1971}
        }
    \bib{KKS}{article}{
        title = {On Harish-Chandra's Plancherel theorem for Riemannian symmetric spaces},
        author = {Krötz, B.},
        author = {Kuit, J. J.},
        author = {Schlichtkrull, H},
        journal = {arXiv:2409.08113}}
    \bib{Mackey}{article}{
        author = {Mackey, G. W.},
        title = {Imprimitivity for Representations of Locally Compact Groups I},
        journal = {Proceedings of the National Academy of Sciences of the United States of America},
        volume = {35, no.9},
        pages = {537-545},
        date = {1949}}
    \bib{MackeyAnalogy}{article}{
        title = {On the analogy between semisimple Lie groups and certain related semidirect product groups},
        author = {Mackey, G. W.},
        pages = {339--363},
        book = {
            title = {Lie groups and their representations (Proc. Summer School, Bolyai János Math. Soc., Budapest, 1971)},
            publisher = {Halsted},
            address = {New York},
            date = {1975}
        }
    }
    \bib{MatsukiOshima}{article}{
        title = {Orbits on affine symmetric spaces under the action of the isotropy subgroups},
        author = {Matsuki, T.},
        author = {Oshima, T.},
        journal = {Journal of the Mathematical Society of Japan},
        volume = {32, no. 2},
        date = {1980},
        pages = {399-414}}
    \bib{MonkVoigt}{article}{
        title = {Complex quantum groups and a deformation of the Baum–Connes assembly map},
        author = {Monk, A.},
        author = {Voigt, C.},
        journal = {Transactions of the American Mathematical Society},
        volume = {371},
        date = {2019},
        pages = {8849--8877}
    }
    \bib{Newstead}{book}{
        title = {Introduction to Moduli Problems and Orbit Spaces},
        author = {Newstead, P. E.},
        publisher = {Narosa},
        series = {Tata Institute of Fundamental Research},
        date = {2012}
        }
    \bib{2authors}{article}{
        title = {An imprimitivity theorem for algebraic groups},
        author = {Parshall, B.},
        author = {Scott, L.},
        journal = {Indagationes Mathematicae (Proceedings)},
        volume = {83, no. 1},
        pages = {39-47},
        date = {1980}}
    \bib{OV}{book}{
        title = {Lie Groups and Algebraic Groups},
        author = {Onishchik, A. L.},
        author = {Vinberg, E. B.},
        publisher = {Springer Berlin, Heidelberg},
        series = {Springer Series in Soviet Mathematics},
        date = {1990}}
    \bib{Rader}{article}{
        title = {Spherical Functions on Cartan Motion Groups},
        author = {Rader, C.},
        journal = {Transactions of the American Mathematical Society},
        volume = {310, no. 1},
        date = {1988},
        pages = {1-45}}
    
    \bib{Raeburn}{article}{
        title = {Induced C*-algebras and a symmetric imprimitivity theorem},
        author = {Raeburn, I.},
        journal = {Mathematische Annalen},
        volume = {280},
        date = {1988},
        pages = {369-387}}
    \bib{VoigtYuncken}{book}{
        title = {Complex Semisimple Quantum Groups and Representation Theory},
        author = {Voigt, C.},
        author = {Yuncken, R.},
        series = {Lecture Notes in Mathematics},
        volume = {2264},
        publisher = {Springer Cham},
        date = {2020}
    }    
    \bib{Wallach}{book}{
        title = {Real reductive groups I},
        author = {Wallach, N.},
        date = {1988},
        publisher = {Academic Press}
    }
    \bib{Warner}{book}{
        title = {Harmonic Analysis on Semi-Simple Lie Groups I},
        author = {Warner, G.},
        publisher = {Springer Berlin, Heidelberg},
        date = {1972}
        }
    \bib{Dana}{book}{
        title = {Crossed products of $C^*$-algebras},
        author = {Williams, D. P.},
        series = {Mathematical Surveys and Monographs},
        volume = {134},
        publisher = {American Mathematical Society},
        address = {Providence, RI},
        date = {2007}
    }
\end{thebibliography}
\end{document}